\documentclass[12pt,a4paper]{amsart}
\usepackage[left=3cm,right=3cm,top=3cm,bottom=3cm]{geometry}

\usepackage[colorlinks=true,linktocpage=true,pagebackref=false,linkcolor=blue]{hyperref}
\usepackage{xcolor}
\usepackage{tikz, tikz-cd}
\usepackage{adjustbox}

\usepackage[cal=euler,scr=boondox]{mathalfa}

\usepackage{amssymb}
\usepackage{amsrefs}

\newcommand\CC{\mathbb{C}}
\newcommand\PP{\mathbb{P}}
\newcommand\ZZ{\mathbb{Z}}
\newcommand\Cc{\mathfrak{C}}

\newcommand\an{\mathrm{an}}

\newcommand\thalf{\tfrac{1}{2}}
\newcommand\Ccc{\mathscr{C}}
\newcommand\DD{\mathscr{D}}
\newcommand\EE{\mathbb{E}}
\newcommand\Ee{\mathscr{E}}
\newcommand\OO{\mathscr{O}}
\newcommand\FF{\mathscr{F}}

\newcommand\SL{\mathrm{SL}}
\newcommand\GL{\mathrm{GL}}
\newcommand\Imm{\mathrm{Im}}
\newcommand\supp{\mathrm{supp}}
\newcommand\Nm{\mathrm{Nm}}

\newcommand\Jac{\mathrm{Jac}}

\newcommand\Spec{\mathrm{Spec}}
\newcommand\UU{\mathrm{U}}
\newcommand\Aa{\mathcal{A}}
\newcommand\BB{\mathcal{B}}
\newcommand\GG{\mathcal{G}}
\newcommand\HH{\mathcal{H}}
\newcommand\MM{\mathcal{M}}
\newcommand\Hh{\mathbb{H}}
\newcommand\uu{\mathfrak{u}}
\newcommand\gl{\mathfrak{gl}}
\newcommand\ssl{\mathfrak{sl}}
\newcommand\sd{\mathrm{sd}}
\newcommand\delbar{\bar{\partial}}
\newcommand\Set{\mathbf{Set}}
\newcommand\Sch{\mathbf{Sch}}
\newcommand\Higgs{\mathbf{Higgs}}
\newcommand\Sh{\mathbf{Sh}}
\newcommand\eps{\varepsilon}
\newcommand\Sym{\mathrm{Sym}}
\newcommand\LL{\mathscr{L}}

\DeclareMathOperator{\End}{End}
\DeclareMathOperator{\Hom}{Hom}
\DeclareMathOperator{\tr}{tr}
\DeclareMathOperator{\diag}{diag}
\DeclareMathOperator{\im}{im}
\DeclareMathOperator{\id}{id}
\DeclareMathOperator{\gr}{gr}

\newtheorem{thm}{Theorem}[section]
\newtheorem{prop}[thm]{Proposition}

\newtheorem{corol}[thm]{Corollary}
\theoremstyle{remark} \newtheorem{rmk}[thm]{Remark}
\theoremstyle{definition} \newtheorem{defn}[thm]{Definition}
\theoremstyle{definition} 
\theoremstyle{definition} 

\author[G. Gallego]{Guillermo Gallego}
\address{Facultad de Ciencias Matemáticas, Universidad Complutense de Madrid, Plaza de Ciencias 3, 28040 Madrid, SPAIN}
\email{guigalle@ucm.es}
\author[O. Garc\'ia-Prada]{Oscar García-Prada}
\address{Instituto de Ciencias Matemáticas, CSIC-UAM-UC3M-UCM, Nicolás Cabrera, 13--15, 28049 Madrid, SPAIN}
\email{oscar.garcia-prada@icmat.es}
\author[M. S. Narasimhan]{M. S. Narasimhan}
\address{Department of Mathematics, Indian Institute of Science, Bangalore - 560012, INDIA \footnote{Professor M.S. Narasimhan passed away on 15 May 2021.}}
\email{narasim@math.tifrbng.res.in}

\thanks{The first author's research is supported by the UCM and Banco Santander under the contract CT63/19-CT64/19. The second author is partially supported by the Spanish MINECO under the ICMAT Severo Ochoa grant No. SEV-2015-0554, and the Spanish MICINN grant No. PID2019-109339GB-C31.}

\subjclass[2020]{Primary 14H60; Secondary 14D23, 53C07}

\title{Higgs bundles twisted by a vector bundle}
\date{} 

\begin{document}
\begin{abstract}
	In this paper, we consider a generalization of the theory of Higgs bundles over a smooth complex projective curve in which the twisting of the Higgs field by the canonical bundle of the curve is replaced by a rank $2$ vector bundle. We define a Hitchin map and give a spectral correspondence. We also state a Hitchin--Kobayashi correspondence for a generalization of the Hitchin's equations to this situation. In a certain sense, this theory lies halfway between the theories of Higgs bundles on a curve and on a higher dimensional variety.
\end{abstract}
\maketitle
\section{Introduction}
Higgs bundles over a smooth complex projective curve $X$ were introduced by Hitchin in \cite{hitchinselfduality} in his study of the
self-duality equations, referred now as Hitchin's equations. A Higgs bundle over $X$  is a pair
$(E,\varphi)$ consisting of a vector bundle $E$ over $X$ and a twisted endomorphism $\varphi:E\to E\otimes K_X$,
where $K_X$ is the canonical bundle of $X$. The moduli space of these objects has an extremely rich geometric structure.
In particular,  the Hitchin map, given by  the coefficients of the characteristic polynomial of $\varphi$, defines a fibration
which is an algebraic completely integrable system --- the Hitchin system. In \cite{hitchinsystem}, Hitchin gives a description
of the generic fibres in terms of spectral curves. These are curves defined in the total space of the line bundle $K_X$. This
construction has been generalized in \cite{bnr} for an arbitrary  twisting line bundle $L$
with $\deg L\geq \deg K_X$. The Hitchin fibration plays a central role in the study of mirror symmetry and Langlands duality
\cites{kapustinwitten,donagipantev}, and has been a crucial ingredient in the proof of the Fundamental Lemma by Ng\^o \cite{ngolemme}.
Higgs bundles on a higher dimensional variety were introduced by Simpson \cite{simpsonhodge}. In this theory the role of the twisting bundle
$K_X$ is taken by the cotangent bundle of the variety, and there is the additional condition  that
$\varphi \wedge \varphi=0$, which is trivally satisfied in dimension $1$. Simpson \cite{simpsonmoduliII} considered also the generalization of the Hitchin map. Very recently, Chen and Ng\^o
\cite{chenngo} have studied this and obtained a number of results in particular in the case of surfaces.

In this paper, we consider a problem that somehow is halfway between the case of curves and the higher dimensional case. While we still work on a complex projective curve $X$, the twisting line bundle $K_X$ is replaced by a higher rank vector bundle $V$. As in higher dimensions, in this situation one naturally considers the integrability condition $\varphi\wedge\varphi=0$. For concreteness and to simplify computations, we limit ourselves here to the case in which both the rank of $E$ and the rank of $V$ are two. This is the simplest non-trivial situation where this problem can be considered.

After introducing the objects and defining stability, we give an algebraic  construction of our moduli space modeled on a construction of the moduli space of Higgs bundles on a higher dimensional variety given by Simpson \cite{simpsonmoduliII}. We also study deformations generalizing the standard approach followed in \cite{biswasramanan}. We then define the Hitchin map in this situation. For simplicity we restrict ourselves to the case in which the determinant of $E$ is trivial and the trace of $\varphi$ vanishes. In Theorem \ref{spectralcorrespondence} we give a spectral correspondence describing the fibres of the Hitchin map. To do this, we follow a procedure inspired by
Chen and Ng\^o's  interpretation of the Hitchin morphism in terms of universal
spectral data \cite{chenngo}. This spectral correspondence allows us to prove that the Hitchin map is proper.  We go on to study for what elements of the Hitchin base the associated spectral curve has nice properties,
like being integral or smooth. The main results here are given by Propositions \ref{integral} and \ref{smooth} and by Corollary \ref{etale}. The question of what conditions can be imposed on the twisting bundle $V$ in order to have generic smooth spectral curves remains open. To complete the circle of ideas relating to the original Higgs bundles introduced by Hitchin, we consider natural gauge-theoretic equations in our situation, generalizing the Hitchin equations and state a Hitchin--Kobayashi correspondence. The existence of solutions follows from the general theorem for twisted quiver bundles proved in \cite{luisoscar}.

The kind of pairs that we study in this paper were already mentioned by Simpson \cite{simpsonmoduliI}*{p.86} and have appeared in the literature sporadically. See, for example the notes \cite{pantev}*{\S 2.2} and the papers of Donagi with Markman \cite{donagimarkman}*{\S 9.3.4} and Gaitsgory \cite{donagigaitsgory}*{\S 17.10}. Very recently, in the work of Alfaya and Oliveira \cite{alfaya}, these are studied in the context of a more general correspondence between $\Lambda$-modules and integrable Lie algebroid connections given by Tortella \cites{tortellathesis,tortella}.

When the twisting vector bundle is a direct sum of two line bundles $V=L_1\oplus L_2$ with $L_1\otimes L_2=K_X$, these objects have been considered by Xie and Yonekura in \cite{xieyonekura}, where they propose the generalized Hitchin's equations as a BPS equation for a certain class of $\mathcal{N}=1$ supersymmetric gauge
theories in four dimensions\footnote{The second author thanks Greg Moore for pointing out this work to him.}. Twisting by a rank two vector bundle appears also in the work of Hitchin \cite{hitchinnahm} on the Nahm equations, regarded as a vector field on the moduli space of co-Higgs bundles on the projective line, as well as in his study \cite{hitchinloci} of singular fibres of the ordinary Hitchin system. When $\det V=K_X$, the total space of $V$ is a noncompact Calabi--Yau threefold and our spectral description may be related to the study of Donaldson--Thomas invariants \cite{donaldsonthomas}. In fact, similar ideas have already been introduced in the study of ADHM sheaves by Diaconescu and others \cites{diaconescu,diaconescuwallcrossing} . 

\section{\texorpdfstring{$V$-} -twisted Higgs bundles}
\subsection{The main definitions}
Let $X$ be a smooth projective curve over $\CC$ and $V\rightarrow X$ a rank $2$ vector bundle.
\begin{defn}
 A \emph{$V$-twisted Higgs bundle} on $X$ is a pair $(E,\varphi)$, for $E\rightarrow X$ a rank $2$ vector bundle and $\varphi:E\rightarrow E\otimes V$ a twisted endomorphism satisfying the \emph{commuting condition} $$\varphi \wedge \varphi=0$$ as a homomorphism $E\rightarrow E\otimes \wedge^2 V$. 
In the course of the text sometimes we will refer to $V$-twisted Higgs bundles simply as \emph{pairs}.

 A \emph{homomorphism} of $V$-twisted Higgs bundles $(E,\varphi)\rightarrow (E',\varphi')$ is a vector bundle homomorphism $f:E\rightarrow E'$ such that the following diagram commutes
\begin{center}
  \begin{tikzcd}
    E \ar{r}{f} \ar{d}{\varphi} & E' \ar{d}{\varphi'}   \\ 
    E \otimes V \ar{r}{f\otimes \id} & E'\otimes V.
  \end{tikzcd}
\end{center}
\end{defn}

\begin{rmk}
	It is clear that with the obvious definitions all the constructions in this section can be generalized to the case where $E$ and $V$ are vector bundles of arbitrary rank.
\end{rmk}

\begin{rmk}
The commuting condition can be better understood in a local frame $\{v_1,v_2\}$ of $V$, so that $\varphi$ can be written as
\begin{equation*}
  \varphi = \varphi_1 \otimes v_1 + \varphi_2 \otimes v_2,
\end{equation*}
where the $\varphi_i$ are local sections of $\End E$. Now,
\begin{equation*}
  \varphi \wedge \varphi = \varphi_1 \circ \varphi_2 \otimes  v_1 \wedge v_2 + \varphi_2 \circ \varphi_1 \otimes v_2 \wedge v_1 = [\varphi_1,\varphi_2] \otimes v_1 \wedge v_2.
\end{equation*}
Thus, locally we can regard $\varphi$ as a pair of commuting endomorphisms. 

This commuting condition is the main aspect where the theory of $V$-twisted Higgs bundles over a curve resembles the theory of usual Higgs bundles in higher dimensions. In that situation, the commuting condition arises naturally as part of the \emph{integrability condition} for the operator $\delbar_E + \varphi$ (see section 1 in Simpson's paper \cite{simpsonlocal} for more details). The commuting condition is also necessary for the existence of a spectral correspondence, since commuting matrices can be simultaneously diagonalized. In this paper we impose the commuting condition in analogy with the higher dimensional Higgs bundle situation, and for the existence of the spectral correspondence, but by doing this we do not mean to imply that pairs $(E,\varphi)$ that do not satisfy the commuting condition might not be worthy of consideration.
\end{rmk}

\begin{defn}
  We say that a $V$-twisted Higgs bundle $(E,\varphi)$ is \emph{stable} (resp. \emph{semistable}) if for every line subbundle  $L\subset E$ such that $\varphi(L)\subset L\otimes V$, we have
  \begin{equation*}
	  \deg(L) < \thalf \deg(E) \ \ \text{ (resp. } \deg(L) \leq \thalf \deg(E)).
  \end{equation*}
\end{defn}

\begin{rmk}\label{tensorline}
	Note that if a pair $(E,\varphi)$ is (semi)stable and $L$ is a line bundle, then the associated pair $(E\otimes L,\varphi)$ (where we identify $\varphi$ with  $\varphi\otimes\id_L$) is also (semi)stable.
\end{rmk}

The following result summarizes some important properties about stable pairs and it can be seen as an analogue of the fact that stable vector bundles are simple. The proof is essentially the same as the one for vector bundles \cite{seshadri}*{p.17, Prop. 6}.

\begin{prop}\label{simple}
	Let $(E_1,\varphi_1)$ and $(E_2,\varphi_2)$ be semistable $V$-twisted Higgs bundles on $X$. Then
	\begin{itemize}
		\item[(a)] If $\deg(E_1)>\deg(E_2)$, then $\Hom((E_1,\varphi_1),(E_2,\varphi_2))=0$.
		\item[(b)] If $E_1$ is stable and $\deg(E_1)=\deg(E_2)=d$, then $\Hom((E_1,\varphi_1),(E_2,\varphi_2))=0$ or $(E_1,\varphi_1)\cong(E_2,\varphi_2)$.
		\item[(c)] If $(E,\varphi)$ is stable, then $\End((E,\varphi))=\CC \cdot \id_E$.
	\end{itemize}
\end{prop}

\subsection{The moduli space} We will construct a \emph{moduli space} $\MM(d)$ that parametrizes equivalence classes of semistable $V$-twisted Higgs bundles of degree $d$ on $X$ based on Simpson's ``second construction'' of the moduli space of (usual) semistable Higgs bundles over projective varieties \cite{simpsonmoduliII}. In particular, there will be an open subvariety $\MM^s(d) \subset \MM(d)$ parametrizing isomorphism classes of stable $V$-twisted Higgs bundles of degree $d$.

\begin{rmk}
	Of course, like in the case of usual Higgs bundles \cite{simpsonmoduliII}*{Lemma 6.5}, a $V$-twisted Higgs bundle is the same that a $\Lambda$-module, for $\Lambda=\Sym^\bullet(V^*)$, and thus one can simply consider the moduli space of $\Lambda$-modules, as constructed in \cite{simpsonmoduliI}*{Thm. 4.7}. This is Simpson's ``first construction" and the approach followed in \cite{alfaya}. In particular, the results of \cite{tortella} give a correspondence between $V$-twisted Higgs bundles and a certain type of ``twisted connections". However, like in \cite{simpsonmoduliII}, for the purpose of this paper (more precisely, for the proof of Proposition \ref{proper}), we prefer to use the ``second construction".   
\end{rmk}

\subsubsection{The moduli functor} We will define a functor describing the moduli problem we have at hand. 
From now on, it will be convenient to fix the usual notation $T_S=T\times_\CC S$, for  $T$ and $S$ schemes over $\CC$. Moreover, if $s\in S$, we denote $T_s=T\times_\CC \{s\}$.

If $S$ is a scheme of finite type over $\CC$, a \emph{family of semistable $V$-twisted Higgs bundles of degree $d$ parametrized by $S$} is a pair $(E,\varphi)$, where $E$ is a vector bundle over $X_S$, flat over $S$, and $\varphi: E \rightarrow E \otimes V_S$ is a morphism with $\varphi \wedge \varphi = 0$, such that, for each closed point $s\in S$, $(E,\varphi)|_{X_s}$ is a semistable $V$-twisted Higgs bundle of degree $d$. A \emph{morphism of families} $(E,\varphi)\rightarrow (E',\varphi')$ is a homomorphism $f:E\rightarrow E'$ such that the following diagram commutes
\begin{center}
  \begin{tikzcd}
    E \ar{r}{f} \ar{d}{\varphi} & E' \ar{d}{\varphi'}   \\ 
    E \otimes V_S \ar{r}{f\otimes \id} & E'\otimes V_S.
  \end{tikzcd}
\end{center}

Consider the functor $\Higgs_d:\Sch/\CC \rightarrow \Set$ sending a $\CC$-scheme $S$ to the set of isomorphism classes of families of semistable $V$-twisted Higgs bundles on $X$ of degree $d$ parametrized by $S$.
We will construct a projective variety $\MM(d)$ that \emph{universally corepresents} the functor $\Higgs_d$, in the sense of Simpson \cites{simpsonmoduliI, simpsonmoduliII}. 
  
\subsubsection{$V$-twisted Higgs bundles as sheaves on $V$} The key idea for the construction of the moduli space is to identify $V$-twisted Higgs bundles with sheaves on $V$. 

Let $Z$ denote a projective completion of $V$ and $D=Z \setminus V$ the divisor at infinity. We can choose $Z$ so that the bundle projection $q:V \rightarrow X$ extends to a map $q:Z \rightarrow X$.

\begin{prop}\label{sheavesonV}
	There exists some $k\in \mathbb{N}$ such that the functor $q_*$ gives an equivalence of categories between families $(E,\varphi)$ of $V$-twisted Higgs bundles of degree $d$ parametrized by $S$ and coherent sheaves $\Ee$ on $Z_S$ of pure dimension $1$ with $\supp(\Ee)\cap D_S=\varnothing$ and Hilbert polynomial $p(\Ee,n)=2kn + d$. 
	In particular, $(E,\varphi)$ is semistable if and only if  $\Ee$ is.
\end{prop}

\begin{proof}
	This is essentially the same as Lemma 6.8 and Corollary 6.9 in \cite{simpsonmoduliII}.
\end{proof}

\subsubsection{Construction of the moduli space} Let $Z$ be the projective completion of $V$ and $k$ as in Proposition \ref{sheavesonV}, and consider the polynomial $P(n)=2kn +d$.

Simpson \cite{simpsonmoduliI}*{Thm. 1.21} showed that the functor $\Sh_{Z,P}: \Sch/\CC \rightarrow \Set$ sending a $\CC$-scheme $S$ to the set of isomorphism classes of families of semistable sheaves over $Z$ with Hilbert polynomial $P$ parametrized by $S$ is universally corepresented by a moduli space $M(Z,P)$. Simpson constructs this moduli space by identifying the set of classes of semistable sheaves of pure dimension with a subset $Q_1$ of a certain Quot scheme and considering a subset $Q_2\subset Q_1$. The moduli space  $M(Z,P)$ is obtained as a good quotient of $Q_2$ by the action of a change of basis.

Let $Q_3\subset Q_2$ the open subset parametrizing quotient sheaves $\Ee$ whose support does not meet $D$. This set $Q_3$ is invariant by change of basis and the inverse image of a subset of $M(Z,P)$, so a good quotient of $Q_3$ exists and it is equal to an open subset $M(V,P)\subset M(Z,P)$. Finally, from the fact that $M(Z,P)$ universally corepresents  $\Sh_{Z,P}$ and from Proposition \ref{sheavesonV}, we conclude that $M(V,P)$ universally corepresents the functor $\Higgs_d$, so we can define the \emph{moduli space of semistable $V$-twisted Higgs bundles of degree $d$ on $X$} as
 \begin{equation*}
	 \MM(d) = M(V,P).
\end{equation*} 

\subsubsection{$S$-equivalence and polystability} We have constructed a moduli space $\MM(d)$ universally corepresenting $\Higgs_d$. However, we still have not said what do the closed points of $\MM(d)$ parametrize. We answer now this question.

Suppose that $(E,\varphi)$ is a semistable  $V$-twisted Higgs bundle that is not stable. Then there exists a line subbundle $L\subset E$ with  $\varphi(L)\subset L \otimes V$ such that  $\deg L = \thalf \deg E$. Now, consider the line bundle  $E/L$ and the induced map $\bar{\varphi}:E/L \rightarrow E/L \otimes V$. We define the \emph{graded $V$-twisted Higgs bundle associated to $(E,\varphi)$} as the direct sum  
\begin{equation*}
	\gr(E,\varphi) = (L,\varphi|_L) \oplus (E/L, \bar{\varphi}).
\end{equation*} 
This object is clearly well defined by $(E,\varphi)$ up to isomorphism. 

When $(E,\varphi)$ is stable, we define  $\gr(E,\varphi)=(E,\varphi)$.
We say that two semistable $V$-twisted Higgs bundles are \emph{$S$-equivalent} if their associated graded objects are isomorphic. We denote by $S(d)$ the set of $S$-equivalence classes of semistable $V$-twisted Higgs bundles of degree $d$ on $X$. Now, it is clear that the subset $S^s(d)$ formed by stable objects is the set of isomorphism classes of stable pairs of degree $d$.

We say that a $V$-twisted Higgs bundle $(E,\varphi)$ is \emph{polystable} if it is stable or if it can be expressed as a direct sum $(E,\varphi)=(L_1,\varphi_1)\oplus (L_2,\varphi_2)$, for $L_i$ line bundles with $\deg L_i=\thalf \deg E$ and $\varphi_i:L_i\rightarrow L_i \otimes V$ with $\varphi_i \wedge \varphi_i=0$. Clearly, a polystable pair is semistable, so the map $[E,\varphi] \mapsto [\gr(E,\varphi)]$ defines a bijection between $S(d)$ and the set of isomorphism classes of polystable pairs of degree $d$.

The natural transformation $\tau:\Higgs_d \rightarrow \underline{\MM(d)}$ corepresenting $\Higgs_d$ induces a map $\tau_\CC:\Higgs_d(\CC)\rightarrow \underline{\MM(d)}(\CC)$, where $\Higgs_d(\CC)$ is the set of isomorphism classes of semistable $V$-twisted Higgs bundles of degree $d$ and $\underline{\MM(d)}(\CC)$ is the set of closed points of $\MM(d)$. 

\begin{prop}
	The map $\tau_\CC$ factors through the quotient map $\Higgs_d(\CC)\rightarrow S(d)$ induced by the $S$-equivalence relation. Moreover, the induced map $S(d)\rightarrow \underline{\MM(d)}(\CC)$ is a bijection.
\end{prop}

\begin{proof}
The proof is standard and identic to others in the literature. See, for example, \cite{simpsonmoduliI}*{Thm. 1.21(3)} and \cite{nitsure}*{Prop. 4.5}.
\end{proof}

As a consequence of this result, we have that the set $S^s(d)$ of isomorphism classes of stable pairs is in bijection with the set of closed points of the open subvariety $\MM^s(d)$ parametrizing stable bundles. This shows that $\MM^s(d)$ is a \emph{coarse moduli space} for stable $V$-twisted Higgs bundles of degree $d$.

  \subsection{Infinitesimal deformations} To finish the section, we give a description of the space of infinitesimal deformations of a $V$-twisted Higgs bundle. Our discussion is a generalization of the results of Biswas and Ramanan \cite{biswasramanan}.

  Consider the ring 
\begin{equation*}
	\CC[\eps]=\CC[X]/(X^2) = \left\{ a+b\eps: a,b \in \CC \right\},
\end{equation*}
which is an Artinian $\CC$-algebra with maximal ideal  $(\eps)$. Given a $\CC$-scheme $S$, we denote $S[\eps]=S\times_\CC \Spec(\CC[\eps])$ and identify $S$ with $S_{(\eps)}$.
Recall that the Zariski tangent space of a variety $M$ at a point $p$ can be identified with the set of maps $\Spec(\CC[\eps]) \rightarrow M$ mapping $(\eps)$ to the point $p$. Therefore, in order to study the ``tangent"   space of the moduli space of $V$-twisted Higgs bundles at a point $(E,\varphi)$, it makes sense to consider the space  $T_{(E,\varphi)}$ of isomorphism classes of families $(\tilde{E},\tilde{\varphi})$ of $V$-twisted Higgs bundles parametrized by  $\Spec(\CC[\eps])$ such that  $(\tilde{E},\tilde{\varphi})|_{X}\cong (E,\varphi)$. We call this $T_{(E,\varphi)}$ the \emph{space of infinitesimal deformations} of $(E,\varphi)$.

\begin{prop}
Let $(E,\varphi)$ be a $V$-twisted Higgs bundle on $X$ and consider the chain complex
\begin{center}
  \begin{tikzcd}
    \Ccc_{(E,\varphi)}: \End E \ar{r}{\varphi \wedge -} & \End E \otimes V \ar{r}{\varphi \wedge -} &  \End E \otimes \wedge^2 V.
  \end{tikzcd}
\end{center}
The space of infinitesimal deformations $T_{(E,\varphi)}$ of $(E,\varphi)$ is canonically isomorphic to the hypercohomology space $\Hh^1(\Ccc_{(E,\varphi)})$.
\end{prop}

\begin{rmk}
	Notice the similarity between the complex $\Ccc_{(E,\varphi)}$ and the deformation complex of an usual Higgs bundle over a projective surface. In particular, the space $\Hh^3(\Ccc_{(E,\varphi)})$ will play an important role in the deformation theory of $(E,\varphi)$. 
\end{rmk}

\begin{rmk}
The above proposition is a particular case of Theorem 47 in the thesis of Tortella \cite{tortellathesis}.	
\end{rmk}

\begin{proof}
  We start by computing $\Hh^1(\Ccc_{(E,\varphi)})$. In order to do this we take the \v{C}ech resolution of the complex $\Ccc_{(E,\varphi)}$, which yields the double complex
  \begin{center}
  \adjustbox{scale=0.9,center}{%
    \begin{tikzcd}
      \End E \ar{r}{\varphi \wedge -} \ar{d} &  \End E \otimes V \ar{r}{\varphi \wedge -} \ar{d} &  \End E \otimes \wedge^2 V
      \End E \ar{d}      \\ 
      K^{0,0}=\sum_i \tilde{M}_i \ar{d} \ar{r} & K^{0,1}=\sum_i \tilde{M}_i \otimes H^0(U_i,V) \ar{d} \ar{r} & K^{0,2}=\sum_i \tilde{M}_i \otimes H^0(U_i,\wedge^2 V) \ar{d} \\
      K^{1,0}=\sum_{ij} \tilde{M}_{ij} \ar{r}   & K^{1,1}=\sum_{ij} \tilde{M}_{ij} \otimes H^0(U_{ij},V) \ar{r}  & K^{1,2}=\sum_{ij} \tilde{M}_{ij} \otimes H^0(U_{ij},\wedge^2 V).
    \end{tikzcd}
  }
\end{center}
This complex has been constructed by covering $X$ by affine open subsets $U_i=\Spec(R_i)$ with affine intersections $U_{ij}=\Spec(R_{ij})$. The bundle $\End E|_{U_i}$ (respectively $\End E|_{U_{ij}}$) can be seen as the sheaf associated to an $R_i$-module $M_i$ (respectively, an $R_{ij}$-module $M_{ij}$). The hypercohomology group $\Hh^1(\Ccc_{(E,\varphi)})$ is now obtained as the cohomology group $H^1(K_{\mathrm{tot}})$ for the total complex $K_{\mathrm{tot}}$ associated to the double complex $K$.

The total complex $K_{\mathrm{tot}}$ is obtained as follows
\begin{align*}
  K_{\mathrm{tot}}^0 &= K^{0,0} = \sum_i \tilde{M}_i, \\
  K_{\mathrm{tot}}^1 &= K^{1,0} \oplus K^{0,1} =  \sum_{ij} \tilde{M}_{ij} \oplus \sum_i \tilde{M}_i \otimes H^0(U_i,V) , \\
  K_{\mathrm{tot}}^2 &= K^{2,0} \oplus K^{1,1} \oplus K^{0,2} \\ & = \sum_{ijk} \tilde{M}_{ijk} \oplus \sum_{ij} \tilde{M}_{ij} \otimes H^0(U_{ij},V) \oplus \sum_{i} \tilde{M}_{i} \otimes H^0(U_i,\wedge^2V), \\
   K_{\mathrm{tot}}^0 &\longrightarrow K_{\mathrm{tot}}^1 \\ 
    s_i &\longmapsto (s_i - s_j,\varphi \wedge s_i) \\
     K_{\mathrm{tot}}^1&\longrightarrow K_{\mathrm{tot}}^2 \\ 
     (s_{ij}, t_{i}) &\longmapsto (s_{ij}+s_{jk}-s_{ik}, \varphi \wedge s_{ij} - (t_i -t_j), \varphi \wedge t_i).
    \end{align*}
    Therefore, $\Hh^1(\Ccc_{(E,\varphi)})$ can be computed as the quotient $Z/B$, where $Z$ consists of pairs $(s_{ij},t_i)$, satisfying the equations
    \begin{equation*}
      \begin{cases}
	s_{ij} + s_{jk} = s_{ij} \\
	t_i- t_j =  \varphi \wedge s_{ij} \\
	\varphi \wedge t_i  =0
      \end{cases}
    \end{equation*}
    and $B$ is the set of pairs of the form $(s_i-s_j,\varphi \wedge s_i)$.
    We will give now an isomorphism $Z/B \rightarrow T_{(E,\varphi)}$.

    For every $i$ we can consider the bundle $\tilde{E}_i=E|_{U_i}[\eps]$, over $U_i[\eps]$. Now, consider an element $(s_{ij},t_i)\in Z$. Then, the restriction $\tilde{E}_i|_{U_j}$ can be identified with $\tilde{E}_j|_{U_i}$ through the isomorphism $\id+s_{ij}\eps$. The first equation above ensures that under these identifications the patches $\tilde{E}_i$ can be glued to give a vector bundle $\tilde{E}$ on $X[\eps] $.

    Over $U_i[\eps]$ we have the section $\tilde{\varphi}_i=\varphi + t_i \eps$ of $\End(\tilde{E}_i)\otimes V[\eps]$. Since the transition functions of $\tilde{E}$ are given by $\id+s_{ij}\eps$, the associated transition functions on $\End(\tilde{E})$ are $\tilde{\varphi}_i \mapsto \tilde{\varphi}_i + s_{ij}\eps \wedge \tilde{\varphi}_i$ . Therefore, the sections $\tilde{\varphi}_i$ will glue to give a global section $\tilde{\varphi}$ of $\End (\tilde{E})\otimes V[\eps]$ if 
    \begin{equation*}
      \tilde{\varphi}_j =\tilde{\varphi}_i + s_{ij} \eps \wedge\tilde{\varphi}_i
    \end{equation*}
    but this is true since
    \begin{equation*}
      \tilde{\varphi}_i + s_{ij} \eps \wedge \tilde{\varphi}_i = \varphi + t_i \eps +  s_{ij}\wedge \varphi \eps = \varphi + t_i \eps - (t_i-t_j) \eps = \varphi + t_j \eps = \tilde{\varphi}_j.  
    \end{equation*}

    Finally, since $\varphi \wedge t_i = 0$, we have
    \begin{equation*}
      \tilde{\varphi}_i \wedge \tilde{\varphi}_i = (\varphi + t_i \eps) \wedge (\varphi + t_i \eps) = \varphi \wedge \varphi + (\varphi \wedge t_i + t_i \wedge \varphi) \eps = 0.
    \end{equation*}
    Therefore, the section $\tilde{\varphi} \in H^0(X,\End(\tilde{E}) \otimes V[\eps])$ satisfies the commuting condition $\tilde{\varphi} \wedge \tilde{\varphi}=0$.

    To sum up, we have defined a map
    \begin{align*}
      Z &\longrightarrow T_{(E,\varphi)} \\ 
      (s_{ij},t_i) &\longmapsto  [(\tilde{E},\tilde{\varphi})].
      \end{align*}

      In order to see that this map descends to a map $Z/B \rightarrow T_{(E,\varphi)}$ we have to check that an element of the form $(s_{ij},t_i)=(s_i-s_j, \varphi \wedge s_i)$, for $(s_i) \in \sum_i \tilde{M}_i$, induces a pair $(\tilde{E},\tilde{\varphi})$ isomorphic to the ``trivial'' pair $(E[\eps],\varphi)$. Indeed, the identification $\tilde{E}_i \rightarrow \tilde{E}_j$ on $U_{ij}$ is given by the automorphism $1+(s_i-s_j)\eps$. Therefore, the following diagram commutes
      \begin{center}
	\begin{tikzcd}
	  \tilde{E}_{ij} \ar{r}{1+s_i \eps} \ar{d}{1+s_{ij}\eps} & \tilde{E}_{ij} \ar{d}{\id}	  \\ 
	  \tilde{E}_{ij} \ar{r}{1+s_j \eps} & \tilde{E}_{ij}.
	\end{tikzcd}
      \end{center}
      This implies that the bundle $\tilde{E}$ is isomorphic to the trivial bundle $E[\eps]$. Moreover, the isomorphism $1+s_i \eps$ sends $\tilde{\varphi}_i$ to
      \begin{equation*}
	\tilde{\varphi}_i + s_i \eps \wedge \tilde{\varphi}_i = \varphi + t_i \eps + s_i \wedge \varphi \eps = \varphi. 
      \end{equation*}

      We thus have defined a map $\Hh^1(\Ccc_{(E,\varphi)}) \rightarrow T_{(E,\varphi)}$. We define now an inverse map $T_{(E,\varphi)} \rightarrow \Hh^1(\Ccc_{(E,\varphi)})$. Let us take then a family $(\tilde{E},\tilde{\varphi})$ of $V$-twisted Higgs bundles parametrized by  $\Spec(\CC[\eps])$, such that $(\tilde{E},\tilde{\varphi})|_{X}\cong (E,\varphi)$. 
      Since the $U_i$ are affine, the restriction $\tilde{E}|_{U_i[\eps]}$ must be of the form $E_i[\eps]$, with $E_i$ a vector bundle over $U_i$. Clearly, $E_i \cong E|_{U_i}$.
      Therefore, $\tilde{E}$ is obtained by gluing the restrictions of $E[\eps]$ to $U_i[\eps]$ via automorphisms over $U_{ij}[\eps]$. These have to be of the form $1+s_{ij} \eps$, with the $s_{ij}$ satisfying the cocycle condition: $s_{ij}+s_{jk}=s_{ik}$ over the triple intersection $U_{ijk}$.

      Moreover, the homomorphism $\tilde{\varphi}$ over each $U_i[\eps]$ must be of the form $\tilde{\varphi}_i=\varphi + t_i \eps$ for sections $t_i \in H^0(U_i,\End E \otimes V)$. The global nature of $\tilde{\varphi}$ imposes the following compatibility condition
\begin{equation*}
  \tilde{\varphi}_i + s_{ij}\eps \wedge \tilde{\varphi}_i = \tilde{\varphi}_j.
\end{equation*}
This implies that $\varphi \wedge s_{ij} = t_i - t_j$.

Finally, from the fact that $\tilde{\varphi}_i \wedge \tilde{\varphi}_i = 0$ we deduce that $\varphi \wedge t_i = 0$. Putting all this together, we see that $(s_{ij},t_i)$ is an element of $Z$. Hence we have given a map $T_{(E,\varphi)}\rightarrow \Hh^1(\Ccc_{(E,\varphi)})$. It is easy to check that this and the map $\Hh^1(\Ccc_{(E,\varphi)})\rightarrow T_{(E,\varphi)}$ constructed above are inverse of each other.
\end{proof}

We can now study the complex $\Ccc_{(E,\varphi)}$ in terms of a simpler complex
\begin{center}
  \begin{tikzcd}
    \DD_{(E,\varphi)}: \End E \ar{r}{\varphi \wedge -} & \End E \otimes V  .
  \end{tikzcd}
\end{center}
This allows us to obtain the following.

\begin{prop}
  If $(E,\varphi)$ is a $V$-twisted Higgs bundle on $X$, we have
  \begin{equation*}
	  \dim \Hh^1(\Ccc_{(E,\varphi)}) - \dim \Hh^2(\Ccc_{(E,\varphi)}) + \dim \Hh^3(\Ccc_{(E,\varphi)}) = \dim \Hh^0(\DD_{(E,\varphi)}).
  \end{equation*}
\end{prop}

\begin{proof}
There is an obvious short exact sequence of complexes
\begin{center}
  \begin{tikzcd}
    0 \ar{r} &    \End E \otimes \wedge^2 V [2] \ar{r} & \Ccc_{(E,\varphi)} \ar{r} & \DD_{(E,\varphi)} \ar{r} & 0.
  \end{tikzcd}
\end{center}
Here, we use the standard notation in derived categories where $[n]$ after a chain complex stands for shifting the complex $n$ places.
This induces a long exact sequence in cohomology
\begin{align*}
  0 & \rightarrow \Hh^1(\Ccc_{(E,\varphi)}) \rightarrow  \Hh^1(\DD_{(E,\varphi)}) \rightarrow  H^0(X,\End E \otimes \wedge^2 V) \\ &\rightarrow \Hh^2(\Ccc_{(E,\varphi)}) \rightarrow \Hh^2(\DD_{(E,\varphi)}) \rightarrow H^1(X,\End E \otimes \wedge^2 V) \rightarrow \Hh^3(\Ccc_{(E,\varphi)}) \rightarrow 0.
\end{align*}
Computing the Euler characteristic we get
\begin{align*}
	\dim \Hh^1(\Ccc_{(E,\varphi)}) &- \dim \Hh^2(\Ccc_{(E,\varphi)}) + \dim \Hh^3(\Ccc_{(E,\varphi)}) \\ &= \dim \Hh^1(\DD_{(E,\varphi)}) - \dim \Hh^2(\DD_{(E,\varphi)}) - \chi(\End E \otimes \wedge^2 V). 
\end{align*}

On the other hand, we have the short exact sequence
\begin{center}
  \begin{tikzcd}
	  0 \ar{r} &    \End E \otimes V [1] \ar{r} & \DD_{(E,\varphi)} \ar{r} & \End E \ar{r} & 0,
  \end{tikzcd}
\end{center}
giving the long exact sequence
\begin{align*}
	0\rightarrow \Hh^0(\DD_{(E,\varphi)}) & \rightarrow H^0(X,\End E) \rightarrow  H^0(X, \End E \otimes V) \rightarrow  \Hh^1(\DD_{(E,\varphi)}) \\ &\rightarrow H^1(X,\End E) \rightarrow H^1(X, \End E \otimes V) \rightarrow \Hh^2(\DD_{(E,\varphi)}) \rightarrow 0.
\end{align*}
Again, computing the Euler characteristic we get
\begin{equation*}
	\dim \Hh^1(\DD_{(E,\varphi)}) - \dim \Hh^2(\DD_{(E,\varphi)}) = \chi(\End E \otimes V) - \chi(\End E) + \dim \Hh^0(\DD_{(E,\varphi)}).
\end{equation*}
Putting this together with the above we have
\begin{align*}
	\dim \Hh^1(\Ccc_{(E,\varphi)}) &- \dim \Hh^2(\Ccc_{(E,\varphi)}) + \dim \Hh^3(\Ccc_{(E,\varphi)}) \\ &= \chi(\End E \otimes V) - \chi(\End E) - \chi(\End E \otimes \wedge^2 V) + \dim \Hh^0(\DD_{(E,\varphi)}). 
\end{align*}
Now, by Riemann--Roch it is easy to see that
\begin{align*}
  \chi(\End E \otimes V) -\chi(\End E)- \chi(\End E \otimes \wedge^2 V) = 0.
\end{align*}
This concludes the proof.
\end{proof}
\begin{rmk}
  Notice that the vanishing of the relation between the Euler characteristics is a very special situation that only occurs when $V$ has rank $2$. 
\end{rmk}

More can be said about the chain complex $\Ccc_{(E,\varphi)}$ when $(E,\varphi)$ is assumed to be stable. 

\begin{prop}
  If $(E,\varphi)$ is a stable $V$-twisted Higgs bundle on $X$, then
  \begin{equation*}
   \Hh^0(\DD_{(E,\varphi)}) \cong \CC.
  \end{equation*}
\end{prop}
\begin{proof}
  It is clear from the relations defining $\Hh^0(\DD_{(E,\varphi)})$ in the \v{C}ech resolution of $\DD_{(E,\varphi)}$ that
  \begin{equation*}
	  \Hh^0(\DD_{(E,\varphi)}) = \left\{ \xi \in H^0(X,\End E) : \varphi \wedge \xi = 0 \right\}=\End(E,\varphi).
  \end{equation*}
  Since $(E,\varphi)$ is stable, by Proposition \ref{simple}, elements $\xi \in \Hh^0(\DD_{(E,\varphi)})$ are of the form $\xi=\lambda \id_E$.
\end{proof}

This proposition implies that, when $(E,\varphi)$ is stable, we have the relation
\begin{equation*}
	  \dim \Hh^1(\Ccc_{(E,\varphi)}) - \dim \Hh^2(\Ccc_{(E,\varphi)}) + \dim \Hh^3(\Ccc_{(E,\varphi)}) = 1.
\end{equation*} 

Finally, assuming some conditions on $V$ we obtain:
\begin{corol}
	Let $K_X$ be the canonical line bundle of  $X$ and denote $L=\wedge^2 V$. Assume that $\deg L \geq \deg K_X$. Then
	 \begin{equation*}
    \dim \Hh^1(\Ccc_{(E,\varphi)})=
    \begin{cases}
	    \dim \Hh^2(\Ccc_{(E,\varphi)}), & \text{if } K_X \cong L, \\
	    \dim \Hh^2(\Ccc_{(E,\varphi)}) + 1, & \text{else}. \\
    \end{cases}
	\end{equation*} 
  Moreover, if $V$ is isomorphic to $V^\vee \otimes K_X$, then there exists a perfect pairing
  \begin{equation*}
    \Hh^1(\Ccc_{(E,\varphi)}) \times \Hh^2(\Ccc_{(E,\varphi)}) \rightarrow \CC.
  \end{equation*}
\end{corol}

\begin{proof}
  Consider the dual complex of $\Ccc_{(E,\varphi)}$,
\begin{center}
  \begin{tikzcd}
    \Ccc_{(E,\varphi)}^*: (\End E)^\vee \otimes L^{-1} \ar{r}{(\varphi \wedge -)^\vee} & (\End E)^{\vee} \otimes V^\vee \ar{r}{(\varphi \wedge -)^\vee} &  (\End E)^{\vee}.
  \end{tikzcd}
\end{center}
By the pairing induced by the Killing form of $\gl(2,\CC)$, we can identify this complex with the complex
\begin{center}
  \begin{tikzcd}
    \Ccc_{(E,\varphi)}^*: \End E \otimes L^{-1} \ar{r}{-(\varphi \wedge -)} & \End E \otimes V^\vee \ar{r}{-(\varphi \wedge -)} &  \End E.
  \end{tikzcd}
\end{center}
By Grothendieck--Serre duality for hypercohomology (see \cite{dalakov}*{\S 3.3} for a comprehensive overview in a similar context), we have
\begin{equation*}
  (\Hh^3(\Ccc_{(E,\varphi)}))^* \cong \Hh^0(\Ccc_{(E,\varphi)}^* \otimes K_X),
\end{equation*}
where $\Ccc_{(E,\varphi)}^*\otimes K_X$ is the chain complex
\begin{center}
  \begin{tikzcd}
    \Ccc_{(E,\varphi)}^* \otimes K_X: \End E \otimes L^{-1} K_X \ar{r}{-(\varphi \wedge -)} & \End E \otimes V^\vee \otimes K_X \ar{r}{-(\varphi \wedge -)} &  \End E\otimes K_X.
  \end{tikzcd}
\end{center}

Now,
\begin{align*}
	\Hh^0(\Ccc_{(E,\varphi)}^*\otimes K_X) &= \left\{ \xi \in H^0(X,\End E \otimes L^{-1}K_X): \varphi \wedge \xi = 0 \right\} \\ & = \Hom((E,\varphi),(E\otimes L^{-1}K_X,\varphi)).
\end{align*}
From Remark \ref{tensorline} we have that, since $(E,\varphi)$ is stable, so will be $(E\otimes L^{-1}K_X,\varphi)$. Now, we assumed that $\deg(K_X)=2g-2\leq \deg(L)$, so $\deg(E\otimes L^{-1}K_X)\leq \deg(E)$.

If $L$ is not isomorphic to $K_X$, Proposition \ref{simple} implies that $$\Hh^0(\Ccc_{(E,\varphi)}^*\otimes K_X)= \Hom((E,\varphi),(E\otimes L^{-1}K_X,\varphi))=0.$$ On the other hand, if $L\cong K_X$ that same proposition asserts that  $$\dim \Hh^0(\Ccc_{(E,\varphi)}^*\otimes K_X)  =1.$$

Finally, Grothendieck--Serre duality gives a perfect pairing $$\Hh^1(\Ccc_{(E,\varphi)}) \times \Hh^2(\Ccc_{(E,\varphi)}^* \otimes K_X) \rightarrow \CC.$$ Now, notice that if $V\cong V^\vee \otimes K_X$, we have that $\Ccc_{(E,\varphi)}^* \otimes K_X = \Ccc_{(E,\varphi)}$, and thus we obtain the desired pairing.
\end{proof}

\section{The Hitchin morphism for \texorpdfstring{$V$-} -twisted \texorpdfstring{$\SL(2,\CC)$-} -Higgs bundles}
\subsection{The Hitchin map}
In this section we introduce the Hitchin morphism and study its fibres by giving a \emph{spectral correspondence}. For simplicity, we restrict to the case where the ``structure group" of the pair is $\SL(2,\CC)$. 

\begin{defn}
	A \emph{$V$-twisted $\SL(2,\CC)$-Higgs bundle} is a $V$-twisted Higgs bundle $(E,\varphi)$ such that $\det E=\OO_X$ and $\tr \varphi =0$.
\end{defn}

The \emph{moduli space $\MM_0$ of semistable $V$-twisted $\SL(2,\CC)$-Higgs bundles} can be obtained from the moduli space of semistable $V$-twisted Higgs bundles of degree $0$ as the fibre of the element $(\OO_X,0)$ by the map
\begin{align*}
	\MM(0) & \longrightarrow \Jac(X) \times H^0(X,V) \\
	(E,\varphi) & \longmapsto (\det E, \tr(\varphi)).
\end{align*} 
Analogously, we can obtain the moduli space $\MM_0^s$ of stable $V$-twisted $\SL(2,\CC)$-Higgs bundles.

The \emph{Hitchin morphism} is defined as the map
\begin{align*}
	h :\MM_0 &\longrightarrow H^0(X,\Sym^2 V)\\ 
	(E,\varphi) &\longmapsto \tr(\varphi^2). 
  \end{align*}
  This map can be interpreted as some sort of ``characteristic polynomial'' since, at least formally
  \begin{equation*}
	  \det(\varphi-\id_E \otimes T) = T^2  - \tr(\varphi^2).
  \end{equation*}
  This is why the elements in the image of $h$ are referred as \emph{spectral data}. We are interested in studying the fibres of this morphism. Note that a good description of these fibres will give a way of constructing semistable $V$-twisted Higgs bundles.

  One of the main differences with the well-known case where the twisting bundle is a line bundle is that in our situation the Hitchin morphism is not surjective. In order to see this, suppose that $V$ splits as a direct sum $V=L_1 \oplus L_2$ (which can always be assumed locally). In that case, we can write $\varphi=(\varphi_1,\varphi_2)$, with $\varphi_i: E\rightarrow E\otimes L_i$. The commuting condition means now that $[\varphi_1,\varphi_2]=0$. The symmetric product $\Sym^2 V$ splits as
  \begin{equation*}
    \Sym^2 V = L_1^2 \oplus L^2_2 \oplus L_1L_2
  \end{equation*}
  and $h(\varphi)$ can be written as
  \begin{equation*}
	  h(\varphi)=\tr(\varphi^2)= (\tr(\varphi_1^2), \tr(\varphi_2^2),\tr(\varphi_1\varphi_2)).
  \end{equation*}
  This implies that an element $b=(b_1,b_2,b_3) \in H^0(X,\Sym^2 V)$ in the image of $h$ must satisfy the relation
  \begin{equation*}
	  b_3^2 = b_1b_2.
  \end{equation*}

  We define the \emph{Hitchin base} as the space
  \begin{equation*}
	  \BB = \left\{ b\in H^0(X,\Sym^2 V) \text{ such that locally } b=(b_1,b_2,b_3) \text{ with } b_3^2=b_1b_2 \right\}.
  \end{equation*}
As we just explained, the Hitchin morphism $h$ factors through the inclusion $\BB \hookrightarrow H^0(X,\Sym^2 V)$.
  In the spirit of \cite{chenngo}, we conjecture that the image of $h$ is precisely the Hitchin base $\BB$.

  \subsection{Universal spectral data}
  We follow a procedure inspired by Chen and Ngô's \cite{chenngo} interpretation of the Hitchin morphism in terms of \emph{universal spectral data}.

  The intuition behind this construction is to consider the ``models over a point'' of each of the elements involved in the Hitchin morphism. For example, over a point a $V$-twisted Higgs bundle is simply a pair $(\varphi_1,\varphi_2)$ of elements of $\ssl(2,\CC)$ with $[\varphi_1,\varphi_2]=0$. We define the \emph{commuting variety} as
  \begin{equation*}
    \Cc = \left\{ (\varphi_1,\varphi_2) \in \ssl(2,\CC): [\varphi_1,\varphi_2]=0 \right\}.
  \end{equation*}
  The point model of the Hitchin morphism is the map
  \begin{align*}
    h :\Cc &\longrightarrow \CC^3 \cong \Sym^2 \CC^2 \\ 
      (\varphi_1,\varphi_2) &\longmapsto (\tr(\varphi_1^2), \tr(\varphi_2^2),\tr(\varphi_1\varphi_2)). 
    \end{align*}
    The image of this map is the cone
    \begin{equation*}
      B = \left\{ (x,y,z) \in \CC^3 : z^2=xy \right\}.
    \end{equation*}

    Now, since an element of $\Cc$ consists of a pair of traceless commuting endomorphisms $\varphi_1$ and $\varphi_2$ of $\CC^2$, there exists a common eigenvector $v\in \CC^2$ of $\varphi_1$ and $\varphi_2$. If $\lambda_i$ denotes the eigenvalue of $\varphi_i$ associated to $v$, we can define the map
    \begin{align*}
       \Cc &\longrightarrow \CC^2 \\ 
        (\varphi_1,\varphi_2) &\longmapsto (\lambda_1,\lambda_2). 
      \end{align*}
      In order to remove the dependence on the choice of the eigenvector $v$, we can consider the quotient $\CC^2/\ZZ_2$ by the action $(\lambda_1,\lambda_2) \mapsto (-\lambda_1,-\lambda_2)$. Composing this quotient with the previous map, we get the \emph{universal spectral data morphism}
      \begin{align*}
	\sd : \Cc &\longrightarrow \CC^2/\ZZ_2 \\ 
	(\varphi_1,\varphi_2) &\longmapsto [(\lambda_1,\lambda_2)] 
	\end{align*}
	Moreover, note that $\CC^2/\ZZ_2$ is in fact isomorphic to the cone $B$ via the map
	\begin{align*}
	  \iota :\CC^2/\ZZ_2 &\longrightarrow \CC^3\\ 
	  [(\lambda_1,\lambda_2)] &\longmapsto (\lambda_1^2,\lambda_2^2,\lambda_1\lambda_2). 
	  \end{align*}
	  Thus, we get a factorization of $h$ as $\iota \circ \sd$. We conclude that in this point model of the Hitchin morphism, spectral data will be given by elements of the cone $B$.

	  \begin{defn}
	    We define the \emph{universal characteristic polynomial} as the map
	    \begin{align*}
	      \chi : \CC^2 \times B &\longrightarrow \CC^3 \\ 
	        \left( (x,y), (b_1,b_2,b_3) \right) &\longmapsto (x^2-b_1,y^2-b_2, xy-b_3). 
	      \end{align*}
	      The \emph{universal spectral cover} is defined as the fibre $\chi^{-1}(0)$ with the natural projection map $p:\chi^{-1}(0)\rightarrow B$.
	  \end{defn}
	  \begin{prop}
	    The universal spectral cover $p:\chi^{-1}(0)\rightarrow B$ is generically étale of degree $2$ over the ``multiplicity-free"  locus $B\setminus \left\{ 0 \right\}$, and it is not flat in general.
	  \end{prop}
	  \begin{proof}
		  If $b\in B\setminus \{0\}$, the corresponding points $(x,y)\in \CC^2$ lying above $b$ are $\pm(\sqrt{b_1},\sqrt{b_2})$. Therefore, the analytification $p^\an$ defines a local biholomorphism, which implies that $p$ is étale of degree $2$ in $b$.

		  On the other hand, the fibre over $0$ is the spectrum of the $\CC$-algebra $$\CC[X,Y]/(X^2,Y^2,XY),$$ which is a $\CC$-vector space of rank $3$. Therefore, $p$ cannot be flat over $0$.
	  \end{proof}

	  Now, we can give the following generalization of the theorem of Cayley--Hamilton.
	  \begin{prop}(Universal Cayley--Hamilton theorem)
	    For any $(\varphi_1,\varphi_2) \in \Cc$, the $\CC[X,Y]$-module induced by the map
	    \begin{align*}
	      \Phi :\CC[X,Y] &\longrightarrow \End \CC^2 \\ 
	        f(X,Y) &\longmapsto f(\varphi_1,\varphi_2), 
	      \end{align*}
	      is supported in the fibre $p^{-1}(h (\varphi_1,\varphi_2))$.
	  \end{prop}
	  \begin{proof}
	    The suport of the module is $\Spec(\CC[X,Y]/\ker \Phi)$. Let $[(\lambda_1,\lambda_2)]$ be the spectral data of $(\varphi_1,\varphi_2)$ and consider the ideal
	    \begin{equation*}
	      I=(X-\lambda_1,Y-\lambda_2) \cdot (X+\lambda_1,Y+\lambda_2) = (X^2-\lambda_1^2,Y^2-\lambda_2^2,XY-\lambda_1\lambda_2).
	    \end{equation*}
	    Clearly $I\subset \ker \Phi$. Therefore, $\Spec(\CC[X,Y]/\ker \Phi)$ is naturally a subscheme of $\Spec(\CC[X,Y]/I) = p^{-1}(h(\varphi_1,\varphi_2))$.
	  \end{proof}

	  \begin{rmk}
	    Notice that if $\lambda_1,\lambda_2 \neq 0$, then in fact $I=\ker \Phi$. This is clear since there exists some basis of $\CC^2$ where the associated matrices to $\varphi_1$ and $\varphi_2$ are $\diag(\lambda_1,-\lambda_1)$ and $\diag(\lambda_2,-\lambda_2)$. Therefore, if $f(\varphi_1,\varphi_2)=0$, in particular we have $f(\lambda_1,\lambda_2)=0$ and $f(-\lambda_1,-\lambda_2)=0$, and thus $f\in I$.

	    This is also true when just one of the eigenvalues (for example $\lambda_2$) vanishes. Indeed, in that case we must have $\varphi_2=0$, but we also have $I=(X^2-\lambda_1^2,Y,XY)$ (notice that $\lambda_1^2 Y = X(XY)$).

	    However, when $[(\lambda_1,\lambda_2)]=[(0,0)]$, in general we can at most hope for a basis where
	    \begin{equation*}
	      \varphi_1 = 
	      \begin{pmatrix}
		0 & 1 \\
		0 & 0
	      \end{pmatrix}
	      \ 
	    \text{ and } \ 
	      \varphi_2 = 
	      \begin{pmatrix}
		0 & a \\
		0 & 0
	      \end{pmatrix}
	      ,
	    \end{equation*}
	    for $a\in \CC$. In this case $f(X,Y)=aX-Y \in \ker \Phi$, but $f\not\in I=(X^2,Y^2,XY)$.
	  \end{rmk}

	  \subsection{The spectral correspondence} We will obtain a spectral curve by twisting all the universal objects constructed in the previous section by the $\GL(2,\CC)$-torsor attached to the vector bundle $V\rightarrow X$. More precisely we take the frame bundle $P$ of $V$, which is a $\GL(2,\CC)$-principal bundle and consider the functor sending any $\CC$-scheme $F$ with a left action of $\GL(2,\CC)$ to the associated fibre bundle $V(F)=P\times_{\GL(2,\CC)} F$ and equivariant morphisms to the associated maps between the associated bundles.
	  
	  The natural action of $\GL(2,\CC)$ on $\CC^2$ induces an action on the cone $\CC^2/\ZZ_2$ and thus on $B$. The sections of the associated fibre bundle $V(B)$ form precisely the Hitchin base 
  \begin{equation*}
    \BB = \left\{ b\in H^0(X,\Sym^2 V) \text{ such that locally } b=(b_1,b_2,b_3) \text{ with } b_3^2=b_1b_2 \right\}.
  \end{equation*}
  Moreover, the universal spectral cover is clearly equivariant by this action, so it induces a fibre bundle morphism $V(p):V(\chi^{-1}(0)) \rightarrow V(B)$.
  
  \begin{defn}
    Given $b\in \BB$, we define the \emph{spectral cover associated to $b$} as the morphism $\pi:Y_b \rightarrow X$ given by the Cartesian diagram
  \begin{center}
    \begin{tikzcd}
      Y_b      \arrow{r}\arrow{d}[anchor=east]{\pi} & V(\chi^{-1}(0))\arrow{d}[anchor=west]{V(p)} \\ 
       X\arrow{r}[anchor=south]{b} & V(B).
     \end{tikzcd}
   \end{center}
   The $1$-dimensional scheme $Y_b$ is called the \emph{spectral curve associated to $b$}.
  \end{defn}
  
  It is clear from this construction that there is a closed embedding $Y_b \hookrightarrow V(\CC^2)=V$. In fact, $\pi$ is obtained as $q|_{Y_b}$, where $q:V\rightarrow X$ is the bundle projection. The finiteness of $p$ implies that the spectral cover $\pi:Y_b \rightarrow X$ is a finite morphism. Moreover, we can consider the subset $\BB'\subset \BB$ consisting of sections mapping the generic point of $X$ to the multiplicity free locus $B\setminus \left\{ 0 \right\}$. If $b\in \BB'$, then $\pi$ is generically étale of degree $2$.
   
  In the case of twisting by a line bundle, studied in \cite{bnr}, the spectral cover is flat as a consequence of the flatness of the universal spectral cover arising in that situation. However, since $p:\chi^{-1}(0)\rightarrow B$ is not flat in general, the spectral cover $\pi$ will not be necessarily flat. If we assume that $b\in \BB'$, then the non-flatness phenomena will show up just in the finite subset of points of $X$ where $b$ vanishes. 

   More precisely, we consider the sheaf $\pi_* \OO_{Y_b}$, which is a coherent sheaf on $X$. Since $X$ is a smooth curve, the sheaf $\pi_* \OO_{Y_b}$ splits as a direct sum $F\oplus T$, where $T$ is a torsion sheaf and $F$ is a locally free sheaf of rank $2$ over $X$. The spectral cover $\pi$ is flat if and only if $\pi_* \OO_{Y_b}$ is locally free or, equivalently, if the torsion sheaf $T$ vanishes. If we assume that $b\in \BB'$, this torsion sheaf $T$ must be supported at the points of $b^{-1}(0)\subset X$. 
   We can define the subscheme $\tilde{Y}_b \subset Y_b$ as the relative spectrum $\tilde{Y}_b = \underline{\Spec}(F)$ and the morphism $\tilde{\pi}=\pi|_{\tilde{Y}_b}$. This morphism $\tilde{\pi}:\tilde{Y}_b\rightarrow X$ is finite locally free, and thus flat.

   As a consequence of \emph{miracle flatness}, non-flatness of the spectral cover $\pi$ amounts to the curve $Y_b$ not being Cohen--Macaulay and thus having some \emph{embedded components}, located precisely at the points of $b^{-1}(0)\subset X$. Geometrically, one can understand the torsion removing process explained above as eliminating the embedded components of $Y_b$, in order to obtain another $1$-dimensional scheme $Y_b$ which is Cohen--Macaulay. This is an example of a \emph{Cohen--Macaulay modification} or \emph{Cohen--Macaulayfication}, easier, but in the same spirit that the one explained in \cite{chenngo}. More details about this process, with explicit computations, are given in Section \ref{cuentas}.

   \begin{rmk}
   It is clear that torsion removing does not affect the irreducible components, so $\tilde{Y}_b$ will be irreducible if and only if $Y_b$ is. Moreover, since $\tilde{\pi}:\tilde{Y}_b \rightarrow X$ is flat, $X$ is irreducible and the generic fibre is reduced, we conclude that the curve $\tilde{Y}_b$ is reduced. Therefore, we have shown that if $Y_b$ is irreducible, then $\tilde{Y}_b$ is integral.
   \end{rmk}

   \begin{thm}[The spectral correspondence] \label{spectralcorrespondence}
     For any $b\in \BB'$ the functor $\pi_*$ gives an equivalence of categories between the category of coherent sheaves $\LL$ on $Y_b$ of generic rank $1$ with $\pi_*\LL$ locally free and $\det \pi_* \LL = \OO_X$ and the category of $V$-twisted Higgs bundles $(E,\varphi)$ with $\tr(\varphi^2)=b$. In particular, this category is nonempty.

     Moreover, if $Y_b$ is irreducible, $\pi_*$ gives a bijection between $h^{-1}(b)$ and the set of equivalence classes of torsion free sheaves $\LL$ on $\tilde{Y}_b$ of generic rank $1$ with $\det\pi_* \LL=\OO_X$.

     Finally, if $\tilde{Y}_b$ is a smooth curve of genus $g'$, the fibre $h^{-1}(b)$ is an abelian variety of dimension $g'-g$, where $g$ is the genus of $X$.
   \end{thm}

   \begin{proof}
     We saw in Proposition \ref{sheavesonV} that a $V$-twisted Higgs bundle $(E,\varphi)$ on $X$ is the same that a coherent sheaf $\Ee$ on $V$ with $E=q_* \Ee$ locally free of rank $2$ and $\det E = \OO_X$.

     The idea now is to apply the universal Cayley--Hamilton theorem. What this result tells us in this situation is that a $V$-twisted Higgs bundle $(E,\varphi)$ is supported as an $\OO_V$-module on the spectral curve $Y_b$, for $b=\tr(\varphi^2)$. Therefore, for $b\in \BB'$, there is a bijective correspondence between $V$-twisted Higgs bundles with $\tr(\varphi^2)=b$ and isomorphism classes of quasi-coherent sheaves $\LL$ on $Y_b$ with $\pi_* \LL$ locally free of rank $2$. Now, since $\pi$ is finite and generically étale of degree $2$, asking that $\pi_* \LL$ has rank $2$ is the same as asking that $\LL$ has generic rank $1$.

     Using the functoriality of taking pushforward and the equivalence of categories given by the closed embedding $\tilde{Y}_b \hookrightarrow Y_b$ we have that, if $(E,\varphi)$ is a $V$-twisted Higgs bundle with $\tr(\varphi^2)=b$, then there is an equivalence of categories between quasi-coherent sheaves $\LL$ on $Y_b$ with $\pi_*\LL=E$ and quasi-coherent sheaves $\LL$ on $\tilde{Y}_b$ with $\tilde{\pi}_*\LL = E$. 
     It suffices then to show that if a sheaf $\LL$ on $Y_b$ satisfies $\pi_* \LL = E$, it is supported on $\tilde{Y}_b$. This is true since the morphism
     \begin{equation*}
       \pi_* \OO_{Y_b} \longrightarrow \End E
     \end{equation*}
     giving the structure of $\pi_* \OO_{Y_b}$-module on $E$ must factor through something which is torsion free, since $\End E$ is torsion free. In particular, it must factor through $\pi_* \OO_{Y_b}/T=F=\tilde{\pi}_* \OO_{\tilde{Y}_b}$. Therefore, $\LL$ is supported on $Y_b$.

     Since $\tilde{\pi}:\tilde{Y}_b \rightarrow X$ is a flat morphism, for any line bundle $L$ on $\tilde{Y}_b$, $\tilde{\pi}_* L$ will be locally free and have the structure of a $V$-twisted Higgs bundle. The determinant of $\tilde{\pi}_* L$ is determined by the formula \cite{hartshorne}*{Ex. IV.2.6(a)}
     \begin{equation*}
       \det \tilde{\pi}_* L \cong \det \tilde{\pi}_* \OO_{\tilde{Y}_b} \otimes \Nm_{\tilde{\pi}} L,
     \end{equation*}
     where $\Nm_{\tilde{\pi}}$ is the norm map of $\tilde{\pi}$.
       Thus it suffices to pick $L$ such that $\Nm_{\tilde{\pi}} L \cong \det (\tilde{\pi}_*\OO_{Y_b})^{-1}$ in order to have $\det \tilde{\pi}_* L \cong \OO_X$.

       Now, if $(E,\varphi)$ is a $V$-twisted Higgs bundle obtained as $\pi_* \LL$ and $L\subset E$ is a $\varphi$-invariant line subbundle, then $\tr(\varphi|_{L}^2)$ must divide $\tr(\varphi^2)$, and $Y_b$ is not irreducible. Therefore, if $Y_b$ is irreducible $(E,\varphi)$ is stable. 
     
       Moreover, it is clear that if $\tilde{Y}_b$ is integral and $\pi_* \LL$ is locally free then $\LL$ must be torsion-free. The same goes for $\tilde{Y}_b$ smooth and $\LL$ locally free. Therefore, if $\tilde{Y}_b$ is smooth, the fibre $h^{-1}(b)$ consists on those line bundles $L$ on $\tilde{Y}_b$ with $\det \tilde{\pi}_* L \cong \OO_X$ but, as we saw above, the set of these bundles is the preimage of $(\det \tilde{\pi}_* \OO_{\tilde{Y}_b})^{-1}$ under the norm map. This is an abelian variety of dimension $g'-g$ (in fact, it can be identified with the \emph{Prym variety} of $\pi$).
   \end{proof}

   One of the main consequences of the spectral correspondence is that it allows us to prove the following.
   \begin{prop}\label{proper}
   	The Hitchin morphism $h:\MM_0 \rightarrow \BB$ is proper.
   \end{prop}

   \begin{proof}
	   We show this by using the valuative criterion. Thus, take $\Delta=\Spec(R)$ the spectrum of a discrete valuation ring, the closed point $p\in \Delta$  and the generic point $\Delta'=\Delta \setminus \{p\}=\Spec(K)$, where $K$ is the quotient field of $R$. Consider a map $g':\Delta'\rightarrow \MM_0$ such that $h\circ g$ extends to a map  $f:\Delta\rightarrow \BB$. We want to find a map $g:\Delta \rightarrow \MM_0$ extending $g'$. 

	   Recall from the construction of $\MM_0$ that this space sits inside the moduli space of sheaves $M(Z,P)$, for $Z$ a projective completion of $V$ and $P(n)=2kn$. Composing with this inclusion, $g'$ defines a map $\Delta'\rightarrow M(Z,P)$ and, since $M(Z,P)$ is a projective variety, this extends to a map  $g:\Delta\rightarrow M(Z,P)$. It suffices to show that $g(p)\in \MM_0$.

	   The good quotient $Q_2\rightarrow M(Z,P)$ is a surjective map of finite type, so we can obtain an integral extension $R\subset \tilde{R}$, where $\tilde{R}$ is another discrete valuation ring over $\CC$ with quotient field $\tilde{K}$ so that, if we denote $\tilde{\Delta}=\Spec(\tilde{R})$, $\tilde{\Delta}'=\Spec(\tilde{K})$ and $\alpha:\tilde{\Delta}\rightarrow \Delta$ the map induced by the extension, the following diagram commutes
	   \begin{center}
		   \begin{tikzcd}
		   \tilde{\Delta}' \ar[hook]{r} \dar&\tilde{\Delta} \ar{d}{\alpha} \ar{r} & Q_2 \ar{d} \\
		   \Delta' \ar[hook]{r}&	   \Delta \ar{r}{g} & M(Z,P). 
		   \end{tikzcd}
	   \end{center}

	   Since $Q_2$ represents a certain functor, the map $\tilde{\Delta}\rightarrow Q_2$ can be identified with a sheaf $\Ee$ on  $Z_{\tilde{\Delta}}$ and, for $s \in \tilde{\Delta}'$, we have that $g\circ \alpha(s) \in \MM_0$. Thus, $\Ee|_{\tilde{\Delta}'}$ is supported inside of $V_{\tilde{\Delta}'}$ and it corresponds to a family $(E',\varphi')$ of semistable $V$-twisted $\SL(2,\CC)$-Higgs bundles on $X$ parametrized by $\tilde{\Delta}'$ with
	   \begin{equation*}
		   (E'_s,\varphi'_s)=g\circ \alpha(s).
	   \end{equation*} 
	   Therefore, the map 
	   \begin{align*}
	   \tilde{\Delta}' & \longrightarrow \BB \\
	   s & \longmapsto h(E'_s,\varphi'_s)
	   \end{align*} 
   extends to a map $b=f\circ \alpha:\tilde{\Delta} \rightarrow \BB$.
	   
	   From the spectral correspondence we have that $\Ee$ is supported inside of the closed subscheme
	   \begin{equation*}
	   W= \coprod_{s \in \tilde{\Delta}} Y_{b(s)} \subset V_{\tilde{\Delta}}.
	   \end{equation*} 
	   Flatness of $\Ee$ over  $\tilde{\Delta}$ implies that $\Ee$ is supported in the closure of $W$, which is again  $W$.
	   Hence $\Ee$ corresponds to a family $(E,\varphi)$ of semistable $V$-twisted Higgs bundles parametrized by $\tilde{\Delta}$, with $(E,\varphi)|_{\tilde{\Delta}'}=(E',\varphi')$. In other words, we obtain a map $\tilde{\Delta}\rightarrow \MM_0$ extending $g'\circ \alpha:\tilde{\Delta}'\rightarrow M(Z,P)$. But since  $M(Z,P)$ is separated, this map coincides with $g\circ \alpha:\tilde{\Delta}\rightarrow M(Z,P)$. This implies that, if $s_0$ is the closed point of  $\tilde{\Delta}$, we have $g(p)=g(\alpha(s_0)) \in \MM_0$, as we wanted to show.
   \end{proof}

   \subsection{Properties of the spectral curve} In the previous section we have obtained a general correspondence between Higgs bundles with spectral data $b$ and sheaves on the corresponding spectral curve $Y_b$, or on its modification $\tilde{Y}_b$. However, we still need to know under what conditions we can assume that the spectral curve $Y_b$ has nice properties, like being integral or smooth.

   First of all, it will be useful to write down explicitly the local equations describing $Y_b$. Given a point $x\in X$, we can take an open neighbourhood $U$ of $x$ so that $V|_U \cong \OO_X^2|_U$ and $b|_U$ can be seen as a section $b=(b_1,b_2,b_3)$ of $\OO_X|_U^3$, with $b_3^2=b_1b_2$. Recall now that the spectral curve is obtained as a relative spectrum $Y_b=\underline{\Spec}(\FF)$, for $\FF=\pi_* \OO_{Y_b}$ a sheaf such that
   \begin{equation} \label{eq:localsheaf}
     \FF|_U = \frac{\OO_X|_U[X,Y]}{(X^2-b_1,Y^2-b_2, XY-b_3)}.
   \end{equation}

   \subsubsection{Integrality} From this description it is clear that $Y_b$ will be reducible if and only if there exists a section $a\in H^0(X,V)$ such that $b=a^2$. Locally, this means that $a=(a_1,a_2)$ and $(b_1,b_2,b_3) =(a_1^2,a_2^2,a_1a_2)$, so
   \begin{equation*}
     (X^2-b_1,Y^2-b_2,XY-b_3) = (X-a_1,Y-a_2)\cdot (X+a_1,Y+a_2).
   \end{equation*}
   In that case, the spectral curve $Y_b$ can be written as $\im a \cup \im (-a)$.

   The image of the natural map $H^0(X,V) \rightarrow \BB$ taking $a$ to $a^2$ is closed. Indeed, we can consider the variety $\PP(\BB)$ as the quotient $\BB\setminus \left\{ 0 \right\}$ by the $\CC^*$ action $b \mapsto \lambda b$. This quotient is a variety since it is the subvariety of $\PP(H^0(X,\Sym^2 V))$ corresponding to $\BB$ (which is locally defined by homogeneous polynomials). Now we can consider the map
   \begin{align*}
      \PP(H^0(X,V)) &\longrightarrow \PP(\BB)\\ 
      [a] &\longmapsto [a^2], 
     \end{align*}
     which is well defined since $(\lambda a)^2 = \lambda^2 a^2$. This is a morphism between projective varieties and as such it has closed image $Z$. Now, the image of the corresponding map $H^0(X,V) \rightarrow \BB$ is the union of $\left\{ 0 \right\}$ with the preimage of $Z$ by the natural projection $\BB \setminus \left\{ 0 \right\} \rightarrow \PP(\BB)$, which is also closed, since this projection is continuous.
   
     We conclude that $Y_b$ is either reducible for every $b\in \BB$, or irreducible for a generic element $b\in \BB$. Moreover, recall that if $Y_b$ is irreducible, then $\tilde{Y}_b$ is integral. Therefore, we have proved:

   \begin{prop}\label{integral}
     If $V$ is such that the map $H^0(X,V) \rightarrow \BB$ given by $a\mapsto a^2$ is not surjective, then for a generic $b\in \BB$, the fibre of the Hitchin morphism $h^{-1}(b)$ is in bijection with the set of equivalence classes of torsion free sheaves $\LL$ on $Y_b$ of generic rank $1$ with $\det \pi_* \LL = \OO_X$.
   \end{prop}

   \subsubsection{Étale spectral curves} We consider now the case where the spectral curves are best behaved, which is when the section $b$ does not vanish.

   \begin{prop}
     If $x\in X$ is such that $b(x) \neq 0$, the induced map $\pi^{\an}: Y_b^\an \rightarrow X^\an$ in the analytifications is a local biholomorphism near any point $y$ with $\pi(y) = x$.
   \end{prop}

   \begin{proof}
     This is clear from \eqref{eq:localsheaf}, since $y$ will be of the form $\pm (\sqrt{b_1(x)},\sqrt{b_2(x)})$, and we can define the local inverse $x' \mapsto \pm (\sqrt{b_1(x')},\sqrt{b_2(x')})$ (making the choice of sign compatible with the value of $y$).
   \end{proof}

   We say that a section $b\in \BB$ is \emph{nonvanishing} if $b(x)\neq 0$ for every $x\in X$.

   \begin{corol}\label{etale}
     If $b\in \BB$ is a nonvanishing section, the spectral cover $\pi:Y_b \rightarrow X$ is étale. In particular, if $Y_b$ is irreducible, the spectral curve $Y_b$ is a smooth projective curve of genus $2g-1$, where $g$ is the genus of $X$. 
     
     Moreover, in that case the fibre of the Hitchin morphism $h^{-1}(b)$ will be an abelian variety of dimension $g-1$.
   \end{corol}

   \begin{proof}
     Since $b$ does not vanish, $\pi^\an$ is a local biholomorphism, which is equivalent to say that $\pi$ is étale.
     If $\pi$ is étale and $X$ is smooth, $Y_b$ must also be smooth. If this $Y_b$ is irreducible, its genus $g'$ is given by the Riemann--Hurwitz formula
     \begin{equation*}
       2g'-2 = \deg \pi (2g-2) ,
     \end{equation*}
     so 
     \begin{equation*}
       g' = 2g -1,
     \end{equation*}
     since $\deg \pi = 2$. 
   \end{proof}

   \begin{rmk}
     Of course, the existence of nonvanishing sections in $\BB$ is not guaranteed a priori. We expect that if $V$ satisfies certain conditions, nonvanishing sections can be found generically, however the question remains open.
   \end{rmk}

   \subsubsection{The general situation} \label{cuentas} We explore now the properties of the spectral curve $Y_b$ and of the modification $\tilde{Y}_b$ in the more general case where $b$ vanishes at some points of $X$. If we assume that $b$ is in the open subset $\BB'\subset \BB$ of sections that do not vanish at the generic point, the set $b^{-1}(0)$ of points where $b$ vanishes will be a finite subset of $X$.

   From the above discussion, it is clear that over a point $x$ outside of $b^{-1}(0)$ the spectral curve $Y_b$ is smooth and thus the spectral cover $\pi$ is flat. Therefore, we are reduced to study the case where $b(x)=0$.

   From \eqref{eq:localsheaf}, it is clear that the stalk of $\FF=\pi_*\OO_{Y_b}$ at a point $x$ is
   \begin{equation*}
   \FF_x = \frac{\OO_{X,x}[X,Y]}{(X^2-b_1,Y^2-b_2,XY-b_3)},
   \end{equation*}
   where the $b_i$ are actually their equivalence classes in $\OO_{X,x}$. This can in principle be decomposed as
   \begin{equation*}
     \FF_x = \OO_{X,x} \oplus X\cdot \OO_{X,x} \oplus Y \cdot \OO_{X,x}.
   \end{equation*}

   When $x\in b^{-1}(0)$, this module $\FF_x$ will not be free, so it will have a decomposition
   \begin{equation*}
     \FF_x = F_x \oplus T_x,
   \end{equation*}
   where $F_x$ is a free module and $T_x$ is a torsion module. Let us find the explicit form of this decomposition.

     Take a small neighbourhood $U$ of $x$ and a local coordinate $z:U\rightarrow \CC$, with $z(x)=0$ and inverse $\phi:\Delta \rightarrow U$, for $\Delta$ a disk centered at $0$, and define the polynomial functions $f_i=b_i \circ \phi: \Delta\rightarrow \CC$. Since $b(x) = 0$, we have that $f_i(0)=0$ for $i=1,2,3$. Therefore, we can write $f_1(Z)= Z^n g_1(Z)$ and $f_3(Z)=Z^m g_3(Z)$, for $g_1$ and $g_3$ not vanishing at $0$. Moreover, by switching $f_1$ and $f_2$, we can assume without loss of generality that $n\geq m$. If we define $a_i = g_i \circ z$, we have
     \begin{align*}
       b_1 &= z^n a_1 \\
       b_3 &= z^m a_3.
     \end{align*}

     We have the relation
     \begin{equation*}
       b_3 X - b_1 Y = (XY) X - X^2Y = 0
     \end{equation*}
     in the module $\FF_x$, so
     \begin{equation*}
       z^m (a_3 X - z^{n-m} a_1 Y) = 0.
     \end{equation*}
     This implies that $a_3 X - z^{n-m} a_1 Y$ is a torsion element of $\FF_x$. Now, the determinant
     \begin{equation*}
       \begin{vmatrix}
	 a_3 & -z^{n-m} a_1 \\
	 z^{n-m} a_1 & a_3
       \end{vmatrix}
       _{x} =
       [a_3^2 + z^{2(n-m)} a_1^2]_x.
     \end{equation*}

     Case $(1)$. If $n> m$ or $a_3(x)^2 + a_1(x)^2 \neq 0$, then this determinant is not $0$,
     so there is an isomorphism
     \begin{equation*}
       X \cdot \OO_{X,x} \oplus Y \cdot \OO_{X,x} \cong [z^{n-m} a_1 X + a_3 Y]\cdot \OO_{X,x} \oplus [a_3X-z^{n-m}a_1 Y] \cdot \OO_{X,x}.
     \end{equation*}
     This gives a decomposition $\FF_x = F_x \oplus T_x$, as a direct sum of a free $\OO_{X,x}$-module
     \begin{equation*}
       F_x = \OO_{X,x} \oplus [z^{n-m} a_1 X + a_3 Y]\cdot \OO_{X,x}
     \end{equation*}
     and a torsion module
     \begin{equation*}
       T_x = [a_3X-z^{n-m}a_1 Y] \cdot \OO_{X,x}.
     \end{equation*}
     Globally, we wrote $\FF=F \oplus T$, with $F=\tilde{\pi}_* \OO_{\tilde{Y}_b}$ a locally free sheaf of rank $2$ and $T$ the torsion part of $\FF$. Now we have given an explicit description of $T$ as a sum of skyscraper sheaves
     \begin{equation*}
       T = \bigoplus_{x \in b^{-1}(0)} i_{x,*} T_x,
     \end{equation*}
     with $i_x: \Spec(\OO_{X,x}) \rightarrow X$ the natural map induced by localization. 
     To sum up, we have obtained that near a point $x\in b^{-1}(0)$, if $n> m$ or $a_3(x)^2 + a_1(x)^2 \neq 0$, the modification $\tilde{Y}_b$ is described by the equations
     \begin{equation*}
       \begin{cases}
	 X^2 - f_1(Z) = 0\\
	 Y^2 - f_2(Z) = 0 \\
	 XY - f_3(Z) = 0 \\
	 g_3(Z) X - Z^{n-m} g_1(Z) Y = 0.
       \end{cases}
     \end{equation*}

     Case $(2)$. If $n=m$ and $a_3(x)^2+a_1(x)^2=0$, we have $a_1(x)=ia_3(x)$, for $i$ a particular choice of the square root of $-1$. Therefore,
      \begin{equation*}
     z^m a_3 (X-iY) = 0,
     \end{equation*} 
     so $a_3(X-iY)$ is a torsion element, and we have an isomorphism
     \begin{equation*}
     X\cdot \OO_{X,x} \oplus Y\cdot \OO_{X,x} \cong a_3 (X+iY)\cdot \OO_X \oplus a_3 (X-iY)\cdot \OO_X,
     \end{equation*} 
     so we can decompose $\FF_x=F_x\oplus T_x$ as a direct sum of a free  $\OO_{X,x}$-module
     \begin{equation*}
     F_x=\OO_{X,x} \oplus a_3 (X+iY)\cdot \OO_X
     \end{equation*} 
     and a torsion module
     \begin{equation*}
     T_x=a_3(X-iY) \OO_{X,x}.
     \end{equation*} 
     In this case we have that the modification $\tilde{Y}_b$ is locally described by the equations
     \begin{equation*}
       \begin{cases}
	 X^2 - if_3(Z) = 0\\
	 Y^2 - f_2(Z) = 0 \\
	 XY - f_3(Z) = 0 \\
	 g_3(Z) (X - iY) = 0.
       \end{cases}
     \end{equation*} 

     The description of $\tilde{Y}_b$ in terms of local equations allows us to give conditions for smoothness over $x$. We just have to apply the Jacobian criterion. We put $F_1(X,Y,Z)=X^2-f_1(Z)$, $F_2(X,Y,Z)=Y^2-f_2(Z)$, $F_3(X,Y,Z)=XY-f_3(Z)$ and $F_4^1(X,Y,Z)=g_3(Z)X-Z^{n-m}g_1(Z)Y$, for case $(1)$, and $F_4^2(X,Y,Z)=g_3(Z)(X-iY)$, for case $(2)$. We denote by $J^1(X,Y,Z)$ the Jacobian matrix of the functions $(F_1,F_2,F_3,F_4^1)$ and by $J^2(X,Y,Z)$ the Jacobian of $(F_1,F_2,F_3,F_4^2)$. The Jacobian criterion asserts that $\tilde{Y}_b$ (if it is irreducible) is smooth at a point $y=\tilde{\pi}^{-1}(x)$ if and only if $J^1(0,0,0)$, in case $(1)$, or $J^2(0,0,0)$, in case $(2)$, has rank greater or equal than $2$, the codimension of $\tilde{Y}_b$ inside the total space of $V$. Now,
     \begin{equation*}
     J^1(0) = \left.
     \begin{pmatrix}
	     2X & 0 & -f_1'(Z) \\ 	
	     0 & 2Y & -f_2'(Z) \\
	     Y & X & -f_3'(Z) \\
	     g_1(Z) & -Z^{n-m}g_1(Z) & (\cdots) Y
     \end{pmatrix} \right|_{0}
     =
     \begin{pmatrix}
	     0 & 0 & -f_1'(0) \\ 	
	     0 & 0 & -f_2'(0) \\
	     0 & 0 & -f_3'(0) \\
	     g_1(0) & 0 & 0
     \end{pmatrix} 
     \end{equation*} 
     and
     \begin{equation*}
     J^2(0) = \left.
     \begin{pmatrix}
	     2X & 0 & -if_3'(Z) \\ 	
	     0 & 2Y & -f_2'(Z) \\
	     Y & X & -f_3'(Z) \\
	     g_3(Z) & -ig_3(Z) & g_3'(Z)(X-iY) 
     \end{pmatrix} \right|_{0}
     =
     \begin{pmatrix}
	     0 & 0 & -if_3'(0) \\ 	
	     0 & 0 & -f_2'(0) \\
	     0 & 0 & -f_3'(0) \\
	     g_3(0) & g_3(0) & 0
     \end{pmatrix} 
     \end{equation*} 
     Both of these matrices have rank $\geq 2$ if and only if one of the $f_i'(0)$ is nonzero; that is, if $m$ equals $1$.

     We have proven:
     \begin{prop}\label{smooth}
       If the modification $\tilde{Y}_b$ is irreducible, then it is smooth if and only if $b$ does not have a multiple zero.
     \end{prop}

     \section{Gauge theoretic equations and Hitchin--Kobayashi correspondence}
     \subsection{The generalized Hitchin equations}
     Higgs bundles twisted by a vector bundle appear as solutions to gauge theoretic equations that are a generalization of the Hitchin equations. These are defined over a compact Riemann surface that we also denote by $X$. 

     We fix once and for all a volume form $\omega_X \in \Omega^{1,1}(X)$ with $\int_X \omega_X=1$.
Consider $V$ a holomorphic vector bundle of rank $2$ over $X$ and fix $h$ a Hermitian metric on $V$. If the vector bundle $V$ is stable, we can choose the solution to the Hermitian Yang--Mills equations given by the theorem of Narasimhan--Seshadri.

Suppose that $(E,\varphi)$ a $V$-twisted Higgs bundle on $X$ and denote
\begin{equation*}
	\lambda_E = -i \pi \deg E \id_E \in \Omega^{0}(X,\End E).
\end{equation*} 
If $H$ is a Hermitian metric on $E$, the twisted endomorphism $\varphi$ has a smooth adjoint morphism $\varphi^\dagger:E\otimes V\rightarrow E$ with respect to the metric induced by $H$ and  $h$, so it makes sense to consider the composition $\varphi^\dagger \circ \varphi\in \Omega^0(X,\End E)$. Moreover, regarding $\varphi^\dagger$ as a smooth map $E\rightarrow E\otimes V^\vee$, we can also consider the composition  $\varphi \circ \varphi^\dagger$ and thus the commutator  $[\varphi,\varphi^\dagger] \in \Omega^0(X,\End E)$. 

\begin{defn}
	Let $(E,\varphi)$ be a $V$-twisted Higgs bundle.	We say that a Hermitian metric $H$ on $E$ is a \emph{solution to the generalized Hitchin equations} if
	\begin{equation*}
	\Lambda F_H + [\varphi,\varphi^\dagger] = \lambda_E,
	\end{equation*} 
	where $F_H$ is the curvature of the Chern connection associated to the metric $H$ on $E$ and $\Lambda$ is the adjoint operator of $\omega_X\wedge -$.
\end{defn}

We have the following \emph{Hitchin--Kobayashi correspondence}:

\begin{prop}\label{hitchinkobayashi}
	A $V$-twisted Higgs bundle $(E,\varphi)$ is polystable if and only if it admits a solution to the generalized Hitchin equations.
\end{prop}
\begin{proof}
	The generalized Hitchin equations are a special case of the \emph{twisted quiver vortex equations}, indroduced in \cite{luisoscar}. More precisely, we can regard $(E,\varphi)$ as a holomorphic $V$-twisted  $Q$-bundle on  $X$, where  $Q$ is the quiver
\begin{center}
  \begin{tikzcd}
    Q= \bullet \ar[loop right, "."]
  \end{tikzcd}
\end{center}
The Hitchin--Kobayashi correspondence for twisted quiver bundles \cite{luisoscar}*{Thm. 3.1} guarantees the existence of the solution $H$, given that $(E,\varphi)$ is polystable, and viceversa.
\end{proof}

\begin{rmk}
	Although the generalized Hitchin equations are a particular case of the more general gauge equations studied in \cite{luisoscar}, in that paper the commuting condition is not considered. However, this extra condition does not affect the Hitchin--Kobayashi correspondence.
\end{rmk}

\subsection{The moment map} Understanding the Hitchin--Kobayashi correspondence as an infinite-dimensional analogue of the Kempf--Ness theorem, $V$-twisted Higgs bundles can be studied in terms of a moment map. As above, we fix a rank $2$ holomorphic vector bundle $V$ on $X$ and a Hermitian metric $h$ on $V$.

We consider  $\EE\rightarrow X$ a rank  $2$ smooth complex vector bundle and fix a Hermitian metric $H$ on  $\EE$. We denote by  $\Aa_H$ the space of  $H$-unitary connections on  $\EE$. Of course, given a connection  $A\in \Aa_H$, the associated operator  $\delbar_A$ endows  $\EE$ with the structure of a holomorphic vector bundle. Finally, we denote
 \begin{equation*}
	 \HH = \Aa_H \times \Omega^0(X,\End \EE\otimes V).
\end{equation*} 

The tangent space of $\HH$ at any point  $(A,\varphi)$ is the vector space $\Omega^1(X,\uu_H(\EE))\times \Omega^0(X,\End \EE \otimes V)$, where $\uu_H(\EE)$ is the bundle of skew-Hermitian endomorphisms of $\EE$ for the metric  $H$. We can define a symplectic on  $\HH$ given by
 \begin{equation*}
  \Omega\left( (B_1,\psi_1), (B_2,\psi_2) \right) = - \int_X \tr(B_1 \wedge B_2) + 2i \Imm(\psi_1^\dagger \psi_2) \omega_X.
\end{equation*} 
The (real) \emph{gauge group} of $\EE$ is $\GG=\Omega^0(X,\UU_H(\EE))$ the group of sections of the bundle $\UU_H(\EE)$ given by $H$-unitary endomorphisms of  $\EE$. The group  $\GG$ acts symplectically by by conjugation on $\HH$. The Lie algebra of  $\GG$  is $\Omega^0(X,\uu_H(\EE))$.
Therefore, the moment map for the action of $\GG$ can be regarded as a map
 \begin{align*}
	 \mu:\HH & \longrightarrow \Omega^0(X,\uu_H(\EE)).
\end{align*} 
The same arguments that those given in \cite{hitchinselfduality}*{\S 4} show that this map is precisely
\begin{equation*}
	\mu(A,\varphi)=\Lambda F_A + [\varphi,\varphi^\dagger].
\end{equation*} 

Now, we can consider the associated symplectic quotient $\mu^{-1}(\lambda_{\EE})/\GG$. The space we are interested in is the subvariety
\begin{equation*}
	\MM_{\mathrm{gauge}}(\EE)= \left\{[(A,\varphi)]\in \mu^{-1}(\lambda_\EE)/\GG: \delbar_{A,V}\varphi = 0 \text{ and } \varphi \wedge \varphi = 0\right\},
\end{equation*} 
where $\delbar_{A,V}=\delbar_A \otimes \id_V + \id_{\End \EE}\otimes \delbar_V$ is the Dolbeaut operator induced on $\End \EE \otimes V$ by the connection $A$ and the holomorphic structure of $V$.

The decomposition $\End(\EE)=\uu_H(\EE) \oplus i \uu_H(\EE)$ induces an isomorphism between $\Omega^{1}(X,\uu_H(\EE))$ and $\Omega^{0,1}(X,\End \EE)$. This allows us to define a complex structure on $\Aa_H$ by $J(\alpha) = i\alpha$. Together with the natural complex structure on $\Omega^0(X,\End\EE \otimes V)$, this gives a complex structure on $\HH$. Moreover, this complex structure is compatible with the symplectic structure $\Omega$ and thus they give a K\"ahler structure. Since the action of the gauge group $\GG$ preserves these structures, the symplectic quotient is in fact a K\"ahler quotient. Therefore, in general $\mu^{-1}(\lambda_{\EE})/\GG$ is a K\"ahler orbifold. Since the conditions defining $\MM_{\mathrm{gauge}}(\EE)$ are analytic, the space $\MM_{\mathrm{gauge}}(\EE)$ is a complex analytic variety.

The Hitchin--Kobayashi correspondence given by the above Proposition \ref{hitchinkobayashi} implies that there is a canonical bijection between the set $S(d)$ of isomorphism classes of polystable $V$-twisted Higgs bundles on $X$ and the space $\MM_{\mathrm{gauge}}(\EE)$. 

Moreover, suppose that $X$ is a smooth complex projective curve, $X^{\mathrm{an}}$ denotes the Riemann surface obtained by analytification of $X$ and $\MM(d)^{\mathrm{an}}$ is the analytification of the moduli space $\MM(d)$ of semistable $V$-twisted Higgs bundles of degree $d$, where $d=\deg \EE$. We expect that the Hitchin--Kobayashi correspondence gives an analytic isomorphism between $\MM(d)^{\mathrm{an}}$ and $\MM_{\mathrm{gauge}}(\EE)$. This is in fact well-known folklore in the field, and similar results are obtained by applying the \emph{Kuranishi slice method}. The reader may refer to \cite{fan} for a detailed exposition of these techniques for the moduli space of (usual) Higgs bundles.

\section*{Acknowledgements}
The first and second authors thank the Indian Institute of Science (Bangalore) and the International Centre for Theoretical Sciences (Bangalore) for hospitality during the visits that led to this collaboration, and during the programme \emph{Moduli of bundles and related structures} (Code: ICTS/mbrs2020/02).

We thank Nigel Hitchin and Tony Pantev for fruitful discussions. We also thank David Alfaya, Peter Dalakov, André Oliveira, and the anonymous referee for their useful comments.

\end{document}